\documentclass[12pt]{article}
\usepackage{amssymb}
\usepackage{amsmath,amsthm}
\usepackage{enumerate}
\usepackage{tikz}
\usepackage{soul}

 \newtheorem{remark}{Remark}

 \newtheorem{lemma}[remark]{Lemma}
 \newtheorem{theorem}[remark]{Theorem}
 
 \newtheorem{corollary}[remark]{Corollary}

\title{On the local metric dimension of corona product graphs}

\author{  Juan A.
Rodr\'{\i}guez-Vel\'{a}zquez,  Gabriel A. Barrag\'{a}n-Ram\'{\i}rez and\\ Carlos Garc\'{i}a G\'{o}mez
    \\
{\small Departament d'Enginyeria Inform\`atica i Matem\`atiques,}\\
{\small Universitat Rovira i Virgili,}  {\small Av. Pa\"{\i}sos
Catalans 26, 43007 Tarragona, Spain.} \\{\small
juanalberto.rodriguez\@@urv.cat, gbrbcn\@@gmail.com, carlos.garciag\@@urv.cat}
}
\date{ }
\begin{document}
\maketitle

\begin{abstract}
A vertex $v\in V(G)$ is said to distinguish two vertices $x,y\in V(G)$ of a nontrivial connected graph $G$ if the distance from $v$ to $x$ is different from  the distance from $v$ to $y$.
 A set $S\subset V(G)$ is a \emph{local metric generator} for $G$ if every two adjacent vertices of $G$ are distinguished by some vertex
of $S$. A local metric generator with the minimum cardinality is called a \emph{local metric
basis} for $G$ and its cardinality, the \emph{local metric dimension} of G.
In this paper we study the problem of finding exact values  for the local metric dimension of corona product of graphs.
\end{abstract}


\section{Introduction}

A generator of a metric space  is a set $S$ of points in the space  with the property that every point of the space is uniquely determined by the distances from the elements of $S$. Given a simple and connected graph $G=(V,E)$, we consider the function $d_G:V\times V\rightarrow \mathbb{R}^+$, where $d_G(x,y)$ is the length of a shortest path between $u$ and $v$. Clearly, $(V,d_G)$ is a metric space, \textit{i.e.}, $d_G$ satisfies $d_G(x,x)=0$  for all $x\in V$, $d_G(x,y)=d_G(y,x)$  for all $x,y \in V$ and $d_G(x,y)\le d_G(x,z)+d_G(z,y)$  for all $x,y,z\in V$. A vertex $v\in V$ is said to distinguish two vertices $x$ and $y$ if $d_G(v,x)\ne d_G(v,y)$.
A set $S\subset V$ is said to be a \emph{metric generator} for $G$ if any pair of vertices of $G$ is
distinguished by some element of $S$. A  metric generator with the minimum cardinality is called a \emph{metric basis}, and
its cardinality the \emph{metric dimension} of $G$, denoted by $\dim(G)$.

Motivated by the problem of uniquely determining the location of an intruder in a network, the concept of metric
dimension of a graph was introduced by Slater in \cite{Slater1975}, where the metric generators were called \emph{locating sets}. The concept of metric dimension of a graph was also introduced by Harary and Melter in \cite{Harary1976}, where metric generators were called \emph{resolving sets}. Applications
of this invariant to the navigation of robots in networks are discussed in \cite{Khuller1996} and applications to chemistry in \cite{Johnson1993,Johnson1998}.  This invariant was studied further in a number
of other papers including, for instance \cite{Bailey2011, Caceres2007, Chartrand2000, Feng20121266, Guo2012raey, Haynes2006, Melter1984, Saenpholphat2004, Yero2011}.  Several variations of metric generators including resolving dominating sets \cite{Brigham2003}, independent resolving sets \cite{Chartrand2003}, local metric sets \cite{Okamoto2010}, strong resolving sets \cite{Sebo2004}, etc. have since been introduced and studied.

In this article we are interested in the study  of local metric generators, also called local metric sets \cite{Okamoto2010}. A set $S$ of vertices
in a connected graph $G$ is a \emph{local metric generator} for $G$ if every two adjacent vertices of $G$ are distinguished by some vertex
of $S$. A local metric generator with the minimum cardinality is called a \emph{local metric
basis} for $G$ and its cardinality, the \emph{local metric dimension} of G, is  denoted by $\dim_l(G)$.
The following main results were obtained in \cite{Okamoto2010}.

  \begin{theorem}{\rm \cite{Okamoto2010}} \label{The1Zhang} Let $G$ be a nontrivial connected graph of order $n$. Then
   $\dim_l(G)=n-1$ if and only if $G$ is  complete, and $\dim_l(G)=1$ if and only if $G$ is bipartite.
\end{theorem}

The clique number $\omega(G)$ of a graph $G$ is the order of a largest complete subgraph in $G$.

  \begin{theorem}{\rm \cite{Okamoto2010}} \label{The1Zhang2} Let $G$ be connected graph of order $n$. Then
   $\dim_l(G)=n-2$ if and only if $\omega(G)=n-1$.
\end{theorem}

In this paper we study the local metric dimension of corona product graphs.
We begin by giving some basic concepts and notations. For two adjacent vertices $u$ and $v$ of $G=(V,E)$ we use the notation  $u\sim v$ and for two isomorphic graphs $G $ and $G'$ we use $G\cong G'$. For a  vertex
$v$ of $G$, $N_G(v)$ denotes the set of neighbors that $v$ has in $G$, \textit{i.e.},  $N_G(v)=\{u\in V:\; u\sim v\}$. The set $N_G(v)$ is called the \emph{open neighborhood of} $v$ in $G$  and $N_G[v]=N_G(v)\cup \{v\}$ is called the \emph{closed neighborhood of} $v$ in $G$.  The degree of a vertex $v$ of $G$ will be denoted by $\delta_G(v)$, \textit{i.e.}, $\delta_G(v)=|N_G(v) |$. Given a set $S\subset V$, we denote by $\langle S\rangle_G$ the subgraph of $G$ induced by $S$ and by $N_G(S)=\cup_{v\in S}N_G(v)$ the \emph{open neighborhood of} $S$. In particular, if $S=\{x\}$ we will use the notation $\langle x\rangle$ instead of $\langle \{x\}\rangle$. 

We will use the notation $K_n$, $K_{r,s}$, $C_n$, $N_n$ and $P_n$ for complete graphs,  complete bipartite graphs, cycle graphs, empty graphs and path graphs, respectively.

Let $G$ and $H$ be two graphs of order $n$ and $n_1$, respectively. The corona product $G\odot H$ was defined by Frucht and Harary in \cite{Frucht1970} as the graph obtained from $G$ and $H$ by taking one copy of $G$ and $n$ copies of $H$ and joining by an edge each vertex from the $i^{th}$-copy of $H$ with the $i^{th}$-vertex of $G$. We will denote by $V=\{v_1,v_2,...,v_n\}$ the set of vertices of $G$ and by $H_i=(V_i,E_i)$ the copy of $H$ such that $v_i\sim x$ for every $x\in V_i$. Figure \ref{Fig} shows two examples  of corona product graphs where the factors are non-trivial.

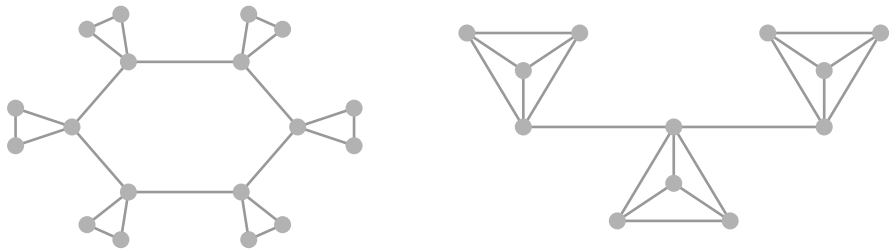
\begin{figure}[h]
\begin{center}
\begin{tikzpicture}
[inner sep=0.7mm, place/.style={circle,draw=black!30,fill=black!30,thick},xx/.style={circle,draw=black!99,fill=black!99,thick},
transition/.style={rectangle,draw=black!50,fill=black!20,thick},line width=1pt,scale=0.5]

\coordinate (X) at (9,0 ); \coordinate (Y) at (13,0 ); \coordinate (Z) at (17,0 );
\coordinate (X1) at (7.5,2.5 ); \coordinate (X2) at (9,1.5 ); \coordinate (X3) at (10.5,2.5 );
\coordinate (Y1) at (11.5,-2.5 ); \coordinate (Y2) at (13,-1.5 ); \coordinate (Y3) at (14.5,-2.5 );
\coordinate (Z1) at (15.5,2.5 ); \coordinate (Z2) at (17,1.5 ); \coordinate (Z3) at (18.5,2.5 );

\coordinate (A1) at (-4.5,0.5);
\coordinate (A2) at (-4.5,-0.5);
\coordinate (A) at (-3,0);
\coordinate (B) at (-1.5,1.732);
\coordinate (B1) at (-2.6,2.6);
\coordinate (B2) at (-1.7,3);
\coordinate (C) at (1.5,1.732);
\coordinate (C1) at (2.6,2.6);
\coordinate (C2) at (1.7,3);
\coordinate (D) at (3,0);
\coordinate (D1) at (4.5,0.5);
\coordinate (D2) at (4.5,-0.5);
\coordinate (E) at (1.5,-1.732);
\coordinate (E1) at (2.6,-2.6);
\coordinate (E2) at (1.7,-3);
\coordinate (F) at (-1.5,-1.732);
\coordinate (F1) at (-2.6,-2.6);
\coordinate (F2) at (-1.7,-3);

\draw[black!40] (X)--(Y)--(Z);
\draw[black!40] (X3)--(X1)--(X)--(X2);
\draw[black!40] (X1)--(X2)--(X3)--(X);

\draw[black!40] (Y3)--(Y1)--(Y)--(Y2);
\draw[black!40] (Y1)--(Y2)--(Y3)--(Y);

\draw[black!40] (Z3)--(Z1)--(Z)--(Z2);
\draw[black!40] (Z1)--(Z2)--(Z3)--(Z);

\draw[black!40] (A)--(B)--(C)--(D) -- (E) -- (F) -- (A);

\draw[black!40] (A1)--(A)--(A2)--(A1);
\draw[black!40] (B1)--(B)--(B2)--(B1);
\draw[black!40] (C1)--(C)--(C2)--(C1);
\draw[black!40] (D1)--(D)--(D2)--(D1);
\draw[black!40] (E1)--(E)--(E2)--(E1);
\draw[black!40] (F1)--(F)--(F2)--(F1);

\node at (X) [place]  {};\node at (X1) [place]  {};\node at (X2) [place]  {};\node at (X3) [place]  {};
\node at (Y) [place]  {};
{};\node at (Y1) [place]  {};\node at (Y2) [place]  {};\node at (Y3) [place]  {};
\node at (Z) [place]  {};
{};\node at (Z1) [place]  {};\node at (Z2) [place]  {};\node at (Z3) [place]  {};
\node at (A) [place]  {};
\node at (A1) [place]  {};
\node at (A2) [place]  {};
\node at (B) [place]  {};
\node at (B1) [place]  {};
\node at (B2) [place]  {};
\node at (C1) [place]  {};
\node at (C2) [place]  {};
\node at (C) [place]  {};
\node at (D1) [place]  {};
\node at (D2) [place]  {};
\node at (D) [place]  {};
\node at (E) [place]  {};
\node at (E1) [place]  {};
\node at (E2) [place]  {};
\node at (F) [place]  {};
\node at (F1) [place]  {};
\node at (F2) [place]  {};
\end{tikzpicture}
\end{center}
\caption{From the left, we show the corona graphs $C_6 \odot K_{2}$ and $P_3\odot K_3$.}\label{Fig}
\end{figure}


The join $G+H$ is defined as the graph obtained from disjoint graphs $G$ and $H$ by taking one copy of $G$ and one copy of $H$ and joining by an edge each vertex of $G$ with each vertex of $H$. Notice that the corona graph $K_1\odot H$ is isomorphic to the join graph $K_1+H$. For instance, the graph $K_1+C_t$ is a wheel graph,  $K_1+K_r\cong K_{r+1}$ and $K_1+N_t$ is isomorphic to a star graph whose central  vertex is the vertex of  $K_1$ and whose $t$ leaves are the vertices of the empty graph $N_t$.   From now on, the vertex of $K_1$ will be denoted by $v$. 

In this work, the remaining definitions will be given the first time that the concept appears in the text.

{The metric dimension and related parameters} have been recently studied for the case of corona graphs. For instance, the metric dimension was studied in \cite{Iswadi} and \cite{Yero2011}, the strong metric dimension was studied in  \cite{Kuziak2013} and the partition dimension was studied in \cite{Rodriguez-Velazquez}. In this article we study  the local metric dimension. The article is organized as follows: In section  \ref{general} we give closed formulae for $\dim_l(G\odot H)$ in terms of $\dim_l(G)$
and $\dim_l(K_1\odot H)$.
Then, we establish lower and upper bounds for $\dim_l(G\odot H)$
by using the orders of $G$ and $H$, and  in Section \ref{extremal} we characterize all graphs when the bounds are attained.
Finally, in Section \ref{radius3} we investigate  the value of $\dim_l(G\odot H)$ when $H$
is a bipartite graph of radius three, and in particular,
we compute $\dim_l(G\odot T)$ when $T$ is a tree.  

\section{General results}\label{general}

To begin with, we consider some straightforward cases.  If  $H$ is an empty graph, then $K_1\odot H$ is a star graph and $\dim_l(K_1\odot H)=1.$ Moreover, if $H$ is a complete graph of order $n$, then $K_1\odot H $ is a complete graph of order $n+1$ and $\dim_l(K_1\odot H)=n.$

\begin{theorem}
Let $G$ be a connected non-trivial graph. For any empty graph $H$, $$\dim_l(G\odot H)=\dim_l(G).$$
\end{theorem}

\begin{proof}
Let $B$ be a local metric basis for $G$. Since in $G\odot H$ every pair of adjacent vertices of $G$ is distinguished by some vertex of $B$ and every vertex  of $B$ distinguishes every pair of adjacent vertices composed by one vertex of $G$ and one vertex of $H$, we conclude that $B$ is a local metric generator for $G\odot H$.

Now, suppose that $A$ is a local metric basis for $G\odot H$ such that $|A|<|B|$. Since $H$ is an empty graph, if there exists $x\in A\cap V_i$, for some $i$, then the pairs of vertices of $G\odot H$ which are distinguished by $x$ can be distinguished also by $v_i$.
So, we consider the set $A'$ obtained from $A$ by replacing by $v_i$ each vertex $x\in A\cap V_i$, where $i\in \{1,...,n\}$.
Thus,  $A'$ is a local metric generator for $G$ and $\vert  A'\vert \le |A|<|B|=\dim_l(G)$, which is a contradiction. Therefore, $B$ is a local metric basis for $G\odot H$.
\end{proof}

We present now the main result on the local metric dimension of corona graphs $G\odot H$ for the case where $H$ is a non-empty graph.

\begin{theorem} \label{mainTheorem}
Let $H$ be a non-empty graph. The following assertions hold.
\begin{enumerate}[{\rm (i)}]
\item  If the vertex of $K_1$ does not belong to any local metric basis for $K_1+H$, then for any connected graph $G$ of order $n$,
$$\dim_l(G\odot H)=n\cdot \dim_l(K_1+H).$$

\item   If the vertex of $K_1$  belongs to a local metric basis for $K_1+H$, then for any connected graph $G$ of order $n\ge 2$,
$$\dim_l(G\odot H)=n  (\dim_l(K_1+H)-1).$$
\end{enumerate}
\end{theorem}

\begin{proof}
If $n=1$, then $G\odot H \cong  K_1+H$ and we are done. We consider $n\ge 2.$  Let $S_i$ be a local metric basis for $\langle v_i \rangle +H_i$ and
let $S_i'=S_i-\{v_i\}$.
Note that $S_i'\ne \emptyset$ because $H_i$ is a non-empty graph   and $v_i$ does not distinguish any pair of adjacent vertices belonging to $V_i$. In order to show that $X=\cup_{i=1}^n S_i'$ is a local metric generator for $G\odot H$ we differentiate the following cases for two adjacent vertices $x,y$.
\\
\\
\noindent Case 1. $x,y\in V_i$. Since $v_i$ does not distinguish  $x,y$, there exists $u\in S_i'$ such that $d_{G\odot H}(x,u)=d_{\langle v_i \rangle +H_i}(x,u)\ne d_{\langle v_i \rangle +H_i}(y,u)= d_{G\odot H}(y,u)$.
\\
\\
\noindent Case 2. $x\in V_i$ and $y=v_i$.  For $u\in S_j'$, $j\ne i$, we have $d_{G\odot H}(x,u)=1+d_{G\odot H}(y,u)> d_{G\odot H}(y,u)$.
\\
\\
\noindent Case 3. $x=v_i$ and $y=v_j$. For $u\in S_j'$,  we have $d_{G\odot H}(x,u)=2=d_{G\odot H}(x,y)+1>1= d_{G\odot H}(y,u)$.

Hence, $X$ is a local metric generator for $G\odot H$.

Now we shall prove (i).  If the vertex of $K_1$ does not belong to any local metric basis for $K_1+H$, then  $v_i\not \in S_i$ for every $i\in \{1,...,n\}$ and, as a consequence,
$$\dim_l(G\odot H)\le |X|=\sum_{i=1}^n |S'_i|=\sum_{i=1}^n \dim_l(\langle v_i \rangle +H_i)=n\cdot \dim_l(K_1 +H).$$
Now we need to prove that $\dim_l(G\odot H)\ge n\cdot \dim_l(K_1 +H) .$ In order to do this, let $W$ be  a local metric basis for $G\odot H$ and let $W_i=V_i\cap W$.
Consider two adjacent vertices $x,y\in V_i-W_i$. Since no  vertex $a\in W-W_i$ distinguishes the pair $x,y$, there exists $u\in W_i$ such that 
$d_{\langle v_i\rangle+ H_i}(x,u)=d_{G\odot H}(x,u)\ne d_{G\odot H}(y,u)=d_{\langle v_i\rangle+ H_i}(y,u)$. So we conclude that $W_i\cup \{v_i\}$ is a local 
metric generator for $\langle v_i\rangle +H_i$. Now, since $v_i$ does not belong to any local metric basis for $\langle v_i \rangle +H_i$, we have that $\vert W_i\vert +1=\vert W_i\cup \{v_i\}\vert>\dim_l(\langle v_i \rangle +H_i)$ and, as a consequence, $|W_i|\ge \dim_l(\langle v_i \rangle +H_i)$. Therefore,   $$\dim_l(G\odot H)= |W|\ge \sum_{i=1}^n |W_i|\ge \sum_{i=1}^n \dim_l(\langle v_i \rangle +H_i)=n\cdot \dim_l(K_1 +H),$$ and the proof of (i) is complete.

Finally,  we shall prove (ii). If the vertex of $K_1$  belongs to a local metric basis for $K_1+H$, then we assume that $v_i\in S_i$ for every $i\in \{1,...,n\}$. Suppose that there exists $B $ such that $B$ is a local metric basis for $G\odot H$ and $|B|<|X|$. In such a case, there exists $i\in \{1,...,n\}$ such that the set $B_i=B\cap V_i $ satisfies $|B_i|<|S_i'|$. Now, since no vertex of $B-B_i$ distinguishes the pairs of adjacent vertices belonging to $V_i$, the set $B_i\cup \{v_i\}$ must be a local metric generator for $\langle v_i \rangle +H_i$. So, $\dim_l(\langle v_i \rangle +H_i)\le  |B_i|+1<\ |S_i'|+1=|S_i|=\dim_l(\langle v_i \rangle +H_i)$, which is a contradiction. Hence, $X$ is a local metric basis for $G\odot H$ and, as a consequence,
$$\dim_l(G\odot H)=|X|=\sum_{i=1}^n |S_i'|=\sum_{i=1}^n (\dim_l(\langle v_i \rangle +H_i)-1)=n( \dim_l(K_1 +H)-1).$$
 The proof of (ii) is now complete.
\end{proof}

As a direct consequence of Theorem \ref{mainTheorem} we obtain the following results.

\begin{corollary}\label{CorollaryExamplesCompleteCompletebipartitePnCn}
The following assertions hold for any connected graph $G$ of order $n\ge 2$.
\begin{enumerate}[{\rm (i)}]
\item For any integer $t\ge 2$, $\dim_l(G\odot K_t)=n(t-1).$
\item For any positive integers $r$ and $s$, $\dim_l(G\odot K_{r,s})=n.$
\item Let $t\ge 4$ be an integer. If $t\equiv 1 (4)$, then $\dim_l(G\odot P_t)=n \left\lfloor\frac{t}{4}\right\rfloor$ and if   $t\not \equiv 1 (4)$, then $\dim_l(G\odot P_t)=n \left\lceil\frac{t}{4}\right\rceil.$
\item For any integer $t\ge 4$, $\dim_l(G\odot C_t)=n \left\lceil\frac{t}{4}\right\rceil$.
\end{enumerate}
\end{corollary}

\begin{proof}
\begin{enumerate}[{\rm (i)}]
\item  If $H\cong K_{t}$, then $K_{1}+K_{t}\cong K_{t+1}$ and the vertex of $K_{1}$ can
belong to a local metric basis for $K_{1}+K_{t}$. Thus,
$$
\dim_{l}\left( G\odot K_{t}\right) =n\cdot \left( \dim_{l}\left(
K_{t+1}\right) -1\right) =n\cdot (t-1).
$$
\item If $H=(U_1\cup U_2,E)\cong K_{r,s}$ then for every $a\in U_1$ (or $a\in U_2$) the set $\{a,v\}$ is a local metric basis for $\langle v\rangle+H$.  Therefore,
$$
\dim_{l}\left( G\odot K_{r,s}\right) =n\cdot \left( \dim_{l}\left(
K_{1}+K_{r,s}\right) -1\right) =n.
$$
\item Notice that a set  $B$ is a local metric basis for $K_{1}+P_{t}$ if and only if for every pair of adjacent vertices $x,y\in V( P_t)$,  vertex $x$ is adjacent to an element of $B$ or vertex $y$ is adjacent to an element of $B$. Thus, for any subgraph $H'$ of $P_t$ isomorphic to $P_4$,  we have $ B\cap V(H')\ne \emptyset$. With this observation in mind, we consider the following two cases.

Case 1. $4\le t\le 5$.
In this case we have that $\dim_l(\langle v\rangle+P_{t})=2$ and $v$ belongs to any local metric basis. Thus,  $\dim_{l}\left( G\odot P_{t}\right)=n=n\left\lfloor \dfrac{t}{4}\right\rfloor$.

Case 2. $t\ge 6$.  For $t=4k+r$, where $0\le r\le 3$, we obtain
\begin{equation}
\dim_{l}\left( K_{1}+P_{t}\right) =\left\{
\begin{array}{c}
k ,\text{ if } r=0\text{ or }r=1 \\
\\
k+1,\text{ if }r=2\text{ or }r=3%
\end{array}%
\right.
\end{equation}%
Therefore, since in this case vertex $v$ does not belong to any local metric basis for $\langle v\rangle +P_{t}$,  we obtain
\begin{equation*}
\dim_{l}\left( G\odot P_{t}\right) =n\cdot \dim_{l}\left(  K_{1}+P_{t}\right) =\left\{
\begin{array}{c}
n\cdot \left\lfloor \dfrac{t}{4}\right\rfloor, \text{  if }t\equiv 1(4) \\
\\
n\cdot \left\lceil \dfrac{t}{4}\right\rceil ,\text{ if }t\not\equiv 1(4).
\end{array}%
\right.
\end{equation*}


\item If $4\le t\le 5$, then $\dim_l(\langle v\rangle+C_{t})=2$.
Since $v$ belongs to any local metric  basis for $\langle v\rangle+C_{4}$ and $v$ does not belong to any local metric basis for $\langle v\rangle+C_{5}$, we have $\dim_l(G \odot C_4)=n$ and $\dim_l(G \odot C_5)=2n=n\left\lceil \dfrac{5}{4}\right\rceil$.

Now we consider the case where $t\ge 6$. As in the proof of (iii), for any local metric basis $B$ of $\langle v\rangle +C_{t}$ and any subgraph $H'$ of $C_t$, isomorphic to $P_4$,  we have $ B\cap V(H')\ne \emptyset$. Hence,  for $t=4k+r$, where $0\le r\le 3$, we deduce
\begin{equation}
\dim_{l}\left( K_{1}+C_{t}\right) =\left\{
\begin{array}{c}
k,\text{ if }r=0 \\
\\
k+1,\text{ otherwise.}%
\end{array}%
\right.
\end{equation}%
 Then, since for $t\ge 6$ vertex $v$ does not belong to any local metric basis for $\langle v\rangle +C_{t}$,
  $\displaystyle \dim_{l}\left( G\odot C_{t}\right) =n\cdot \dim_{l} \left( K_{1}+C_{t}\right) =n\cdot \left\lceil \dfrac{t}{4}\right\rceil
.$
\end{enumerate}

\end{proof}

{
Since any metric generator is a local metric generator,  the local metric dimension of a graph $G$ is at most equal to the metric dimension of $G$, \textit{i.e.}, $\dim_l(G)\le \dim(G)$. For instance, for the complete graph of order $n\ge 2$,  $\dim_l(K_n)= \dim(K_n)=n-1$, and for any bipartite graph $G$, different from a path, $\dim_l(G)=1< \dim(G)$.  As an illustrative  example where the local metric dimension can be significantly smaller than the metric dimension, we can take the complete bipartite graph $K_{r,s}$ of order $r+s\ge 4$, where  $\dim_l(K_{r,s})=1< r+s - 2= \dim(K_{r,s})$.
 Similar examples can be derived for corona graphs. For instance, it was shown in \cite{Yero2011} that for any connected  graph $G$ of order $n\ge 2$ and any integers $r\ge 2$ and $s\ge 1$ ($r+s\ge 4$),
 $\dim(G\odot K_r) = n(r-1)$ and 
  $\dim (G\odot K_{r,s})=n(r+s-2).$ Thus, according to
 Corollary \ref{CorollaryExamplesCompleteCompletebipartitePnCn} (i) and (ii),
 $\dim_l(G\odot K_r) =n(r-1)= \dim (G\odot K_r)$ and 
  $\dim_l(G\odot K_{r,s}) =n< n(r+s-2)=\dim (G\odot K_{r,s}).$}

\begin{corollary}
For any connected graph $H$ and any connected graph $G$ of order $n\ge 2$,
$$\dim_l(G\odot H)\ge n\cdot \dim_l(H).$$
\end{corollary}

\begin{proof}
Let $B$ be a local metric basis for $K_1+H$. Since the vertex $v$ of $K_1$ does not distinguish any pair of adjacent vertices $x,y\in V(H)$, $B-\{v\}$ is a local metric generator for $H$. Thus, if $v\in B$, then $\dim_l(K_1+H)-1\ge \dim_l(H)$ and, if $v\not \in B$, then $\dim_l(K_1+H)\ge \dim_l(H)$. Therefore, Theorem \ref{mainTheorem} leads to $\dim_l(G\odot H)\ge n\cdot \dim_l(H).$
\end{proof}

Now we will give some results involving the diameter or the radius of $H$. The \textit{eccentricity} $\epsilon(v)$ of a vertex $v$ in a connected graph $H$ is the maximum graph distance between $v$ and any other vertex $u$ of $H$. So, the \textit{diameter} of $H$ is defined as $$D(H)=\displaystyle\max_{v\in V(H)}\{\epsilon(v)\},$$ while the \textit{radius}  is defined as $$r(H)=\displaystyle\min_{v\in V(H)}\{\epsilon(v)\}.$$

\begin{corollary}
For any graph $H$ of diameter two and any connected graph $G$ of order $n\ge 2$,
$$\dim_l(G\odot H)=n\cdot \dim_l(H).$$
\end{corollary}

\begin{proof}
Since $H$ has diameter two, for every $x,y\in V(H)$ it follows  $d_H(x,y)=d_{K_1+H}(x,y)$. So, if the vertex of $K_1$ does not belong to any local metric basis for $K_1+H$, then every local metric basis for $H$ is a local metric basis for $K_1+H$ and vice versa. Hence, in such a case,  Theorem \ref{mainTheorem} (i) leads to $\dim_l(G\odot H)=n\cdot \dim_l(H).$

Now we suppose that there exists a local metric basis $B$ of $K_1+H$ such that the vertex $v$ of $K_1$ belongs to $B$. Since $v$ does not distinguish any pair of vertices of $H$, $B'=B-\{v\}$ is a local metric generator for $H$. Moreover, if there exists $A \subset V(H)$ such that $|A|<|B'|$ and $A $ is a  local metric basis for $H$, then $A\cup \{v\}$ is a local metric generator for $K_1+H$, which is a contradiction because $|A|+1<|B'|+1=|B|=\dim_l(K_1+H)$. Therefore, $B'$ is a  local metric basis for $H$ and, as a result, $\dim_l(K_1+H)=1+\dim_l(H)$. So, by Theorem \ref{mainTheorem} (ii) we obtain  $\dim_l(G\odot H)=n\cdot \dim_l(H).$
\end{proof}

\begin{lemma}\label{Lemmak1belonsl}
Let $H$ be a graph of radius $r(H)$. If $r(H)\ge 4$ 
then the vertex of $K_1$ does not belong to any local metric basis for $K_1+H.$
\end{lemma}

\begin{proof}
Let $B$ be a local metric basis for $K_1+H$. We suppose that the vertex $v$ of $K_1$  belongs to $B$. Note that $v\in B$ if and only if there exists $u\in V(H)-B$ such that $B\subset N_{K_1+H}(u)$.

Now, if $r(H)\ge 4$, then we take  $u'\in V(H)$ such that $d_H(u,u')=4$   and  a shortest path $uu_1u_2u_3u'$. In such a case for every $b\in B-\{v\}$ we will have that $d_{K_1+H}(b,u_3)=d_{K_1+H}(b,u')=2$,  which is a contradiction.  Hence, $v$ does not belong to any local metric basis for $K_1+H$. 
\end{proof}

The converse of Lemma \ref{Lemmak1belonsl} is not true. In Figure \ref{K1+H} we show a graph $H$  of radius three where the vertex of $K_1$ does not belong to any local metric basis for $K_1+H$.

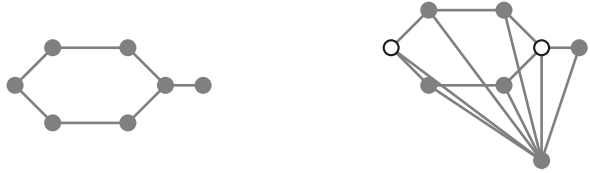
\begin{figure}[h]
\begin{center}
\begin{tikzpicture}
[inner sep=0.7mm, place/.style={circle,draw=black!50,fill=black!50,thick},xx/.style={circle,draw=black!90,fill=white!90,thick},
transition/.style={rectangle,draw=black!50,fill=black!20,thick},line width=1pt,scale=0.5]
\coordinate (A) at (-8,0); \coordinate (B) at (-7,1);
\coordinate (C) at (-7,-1); \coordinate (D) at (-5,-1);
\coordinate (E) at (-5,1); \coordinate (F) at (-4,0);
\coordinate (G) at (-3,0); 
\draw[black!50] (A) -- (B) -- (E) -- (F)--(D)--(C)--(A);
\draw[black!50]  (F) -- (G);
\node at (A) [place]  {};\node at (B) [place]  {};
\node at (C) [place]  {};\node at (D) [place]  {};
\node at (E) [place]  {};
\node at (F) [place]  {};\node at (G) [place]  {};
\coordinate (A1) at (2,1); \coordinate (B1) at (3,2);
\coordinate (C1) at (3,0); \coordinate (D1) at (5,0);
\coordinate (E1) at (5,2); \coordinate (F1) at (6,1);
\coordinate (G1) at (7,1); 
\coordinate (V1) at (6,-2);
\draw[black!50] (A1) -- (V1) -- (B1);\draw[black!50] (E1) -- (V1) -- (F1);\draw[black!50] (C1) -- (V1) -- (D1);
\draw[black!50] (G1) -- (V1);
\draw[black!50] (A1) -- (B1) -- (E1) -- (F1)--(D1)--(C1)--(A1);
\draw[black!50] (F1) -- (G1);
\node at (A1) [xx]  {};\node at (B1) [place]  {};
\node at (C1) [place]  {};\node at (D1) [place]  {};
\node at (E1) [place]  {};
\node at (F1) [xx]  {};\node at (G1) [place]  {};
\node at (V1) [place]  {};
\end{tikzpicture}
\end{center}
\vspace{-0,8cm}
\caption{A graph $H$ and the join graph $K_1+H$.  White vertices form a local metric basis for $K_1+H$. }\label{K1+H}
\end{figure}

The following result is a direct consequence of  Theorem  \ref{mainTheorem} (i) and Lemma \ref{Lemmak1belonsl}.
 \begin{theorem}\label{Th-rM4}
 For any    connected graph $G$ of order $n$ and any graph $H$ of radius $r(H)\ge 4$,
 $$ \dim_l(G\odot H)= n\cdot \dim_l(K_1+H).$$
 \end{theorem}

Another consequence of Theorem \ref{mainTheorem} is the following result.

\begin{corollary}\label{bound-orders}
  For any non-empty graph   $H$  of order {$n'\ge 2$}   and any connected graph   $G$  of order $n\ge 2$,
$$n\le \dim_l(G\odot H)\le n(n'-1).$$
\end{corollary}

The aim of the next section is the study of the limit cases of Corollary \ref{bound-orders}.

\section{Extremal values}\label{extremal}
\begin{theorem}
Let $H$ be a graph  of order $n'$ and let   $G$  be a connected graph of order $n\ge 2$. Then
$\dim_l(G\odot H)=n(n'-1)$ if and only if $H\cong K_{n'}$. 
\end{theorem}

\begin{proof}By Theorem \ref{mainTheorem} we conclude that $\dim_l(G\odot H)=n(n'-1)$ if and only if exactly one of the following cases hold:
 \\
 \noindent Case \emph{a}: the vertex $v$ of $K_1$ does not belong to any local metric basis for $K_1+H$ and $\dim_l(K_1+H)=n'-1$.
\\
 \noindent Case \emph{b}: the vertex $v$ of $K_1$  belongs to a local metric basis for $K_1+H$ and $\dim_l(K_1+H)=n'$.  \\

 We first consider Case \emph{a}. By  Theorem \ref{The1Zhang2}  $\dim_l(K_1+H)=n'-1$ if and only if $\omega(H)=n'-1$.  Let $V(H)=\{u_1,u_2,...u_{n'}\}$. If $\langle V(H)-\{u_1\}\rangle$ is a clique and  $u_i\sim u_1$, then $  \{v\}\cup V(H)-\{u_1,u_i\}$ is a local metric basis for $K_1+H$, which is a contradiction. Hence  $u_1$ is an isolated vertex of $H$. So, $H\cong K_1\cup  K_{n'-1}$, which is a contradiction, as $\{v,u_3,\ldots , u_{n'}\}$ is a local metric basis of $\langle v\rangle+H$.


Finally, by Theorem \ref{The1Zhang} we deduce that Case \emph{b} holds if and only if $H\cong K_{n'}$.
\end{proof}

The radius $r(G)$ of a graph $G$ is the minimum eccentricity of any  vertex of $G$. The center of $G$, denoted by $C(G)$, is the set of vertices of $G$ with eccentricity equal to $r(G)$.

\begin{theorem}\label{Th-r=2Bipartite}
Let $H$ be a non-empty graph and let $G$   be a connected graph  of order $n\ge 2$. Then
$\dim_l(G\odot H)=n$ if and only if $H$ is a bipartite graph having only one non-trivial connected component $H^*$ and $r(H^*)\le 2$.
\end{theorem}

\begin{proof}
Since $\langle v\rangle +H $ is not bipartite, by Theorem \ref{The1Zhang} we deduce $\dim_l(\langle v\rangle +H )\ge 2$. So, if $\dim_l(G\odot H)=n$, then by Theorem \ref{mainTheorem} we have that $\dim_l(\langle v\rangle +H )=2$ and $v$ belongs to a local metric basis for $\langle v\rangle +H $,  say $B=\{u,v\}$. So,  $B\cap V(H)=\{u\}$ must be a local metric generator for $H$ and, by Theorem \ref{The1Zhang}, we conclude that $H$ is a bipartite graph having only one non-trivial connected component. Moreover, if the non-trivial component of $H$ has radius  $r>2$, then there exists $u_3\in V(H)$ such that $d_H(u,u_3)=3$ and, as a consequence, for any shortest path $uu_1u_2u_3$ we have  $d_{\langle v\rangle +H }(u,u_2)=d_{\langle v\rangle +H }((u,u_3)$, \textit{i.e.}, the pair of adjacent vertices $u_2,u_3$ is  not distinguished by the elements of $B $, which is a contradiction. Therefore,  $r\le 2$.

Conversely, let $H$ be a bipartite graph where having only one non-trivial component $H^*$.  Let  $r(H^*)\le 2$,  let $a$ be a vertex belonging to the center of $H^*$ and let $v$ be the vertex of $K_1$. Since $H$ is a triangle free graph,  $a$ distinguishes every pair of adjacent vertices $x,y\in V(H^*)$. So, $\{v,a\}$ is a local metric generator for $K_1+H$, which is a local metric basis because $\dim_l(K_1+H)\ge 2$. We conclude the proof by Theorem \ref{mainTheorem} (ii).
\end{proof}

\section{The value of $\dim_l(G\odot H$) when $H$ is a bipartite graph of radius three}\label{radius3}

{Theorems \ref{Th-rM4} and \ref{Th-r=2Bipartite} suggest to consider the case where $H$ is a bepartite graph of radius three}.
To do that, we need the following additional notation. For any $a\in V(H)$, we denote $$N_H^{(i)}(a)=\{w\in V(H): \; d_H(w,a)=i\}.$$ We also define $N_H^{(i)}[a]=N_H^{(i)}(a)\cup \{a\}$. Note that $N_H^{(1)}(a)=N_H(a)$ and $N_H^{(1)}[a]=N_H[a]$. Given two sets $A,B\subset V(H)$ we say that $A$ dominates $B$ if every vertex in
$B - A$ is adjacent  to some vertex belonging to $A$. From now on we will use the notation $A\succ B$ to indicate that $A$ dominates $B$. 
 For every $x\in C(H)$, let $\beta(x)=\min \left\{|A|:\; A\subseteq  N_H(x)\; {\rm and}\; A\succ N_H^{(2)}(x)\right\}$ and  let
 $$\delta'(H)=\min_{x\in C(H)}\left\{\beta(x)\right\}.$$
\begin{lemma}\label{lemmaRadio3Bipartite}
For any  bipartite graph $H$ of radius three,  $$\dim_l(K_1+H)\le \delta'(H)+1.$$ Moreover,  $\dim_l(K_1+H)= \delta'(H)+1$ if and only if the vertex of $K_1$ belongs to a local metric basis for $K_1+H$.
\end{lemma}

\begin{proof}Let $u$ be a vertex belonging to the center of $H$ and $ A\subseteq  N_H(u)$ such that $ A\succ N_H^{(2)}(u)$ and $|A|=\delta'(H)$. Let us show that $B=A\cup \{v\}$ is a local metric generator for $\langle v \rangle+H$.  We first note that  since $H$ is bipartite, for two adjacent vertices $x,y\not\in B$ it follows $d_H(u,x)\ne d_H(u,y)$.
Hence, without loss of generality, we may consider the following three cases for two adjacent vertices $x,y\not\in B$.

\noindent Case 1: $x=u$ and $y\sim u$.  In this case for every $z\in A$ it follows $d_{K_1+H}(x,z)=1$ and $d_{K_1+H}(y,z)=2$.\\

\noindent Case 2:  $d_H(u,x)=1$ and $d_H(u,y)=2$. In this case $y\in N_H^{(2)}(u)$ and there exists $x'\in A$ such that $x'\sim y$ and, since $H$ is a bipartite graph,  $x' \not\sim x$. So, $d_{K_1+H}(x,x')=2$ and $d_{K_1+H}(y,x')=1$.\\

\noindent Case 3:  $d_H(u,x)=2$ and $d_H(u,y)=3$. In this case $x\in N_H^{(2)}(u)$ and there exists $x'\in A$ such that $ux'xy$ is a shortest path in $H.$ So, $d_{K_1+H}(x,x')=1$ and $d_{K_1+H}(y,x')=2$.\\

Thus,  $B$ is a local metric generator for  $K_1+H$ and, as a consequence, $\dim_l(K_1+H)\le \delta'(H)+1.$

Moreover, if $\dim_l(K_1+H)= \delta'(H)+1$, then $B$ is a local metric basis for $K_1+H$ which contains the vertex of $K_1$.

Conversely, let $S$ be  a local metric basis for $K_1+H$ which contains the vertex  $v$ of $K_1$. In this case there exists $w\in V(H)$ such that $N_H(w)\supset S-\{v\}$. If $w\not\in C(H)$, then there exists $w'\in V(H)$ such that $d_H(w,w')\ge 4$ and for every shortest path $ww_1w_2w_3w'$ from $w$ to $w'$  the pair of vertices $w_3,w'$ is not resolved in $K_1+H$ by any $s\in S$, which is a contradiction. Hence, $w\in C(H)$ and $S-\{v\}\succ N_H^{(2)}(w)$. The minimality of the cardinality of $S$ leads to  $|S-\{v\}|=\delta'(H).$
 Therefore, $\delta'(H)+1=|S|=\dim_l(K_1+H)$.
\end{proof}

As a direct consequence of Theorem \ref{mainTheorem}  and Lemma \ref{lemmaRadio3Bipartite} we obtain the following result.

\begin{theorem}
  Let $H$ be a bipartite graph  of radius three and let $G$ be a connected graph of order $n\ge 2$. Then
  $$\dim_l(G\odot H)\le n\cdot \delta'(H).$$
\end{theorem}


\subsection{The maximum value}
In this section we show that the above bound is attained for a subfamily  of bipartite graphs   of diameter three that does not contain a square (a subgraph  isomorphic to $K_{2,2}$).  In such a case, the girth of $H$ must be six and $H=(U_1\cup U_2,E)$ satisfies the following property:
\begin{enumerate}[$\blacklozenge$]
\item For any $i\in \{1,2\}$ and any two distinct vertices $a,b\in U_i$,  $\vert N_H(a)\cap N_H(b)\vert=1$.
\end{enumerate}
Therefore, $H$ is the incidence graph of a finite projective plane.   So, we have two possibilities (see, for instance, \cite{BondyBook2008}):

\begin{enumerate}[(P1)]

\item $H=(U_1\cup U_2,E)$ is the incidence graph of a degenerate projective plane. In this case $\vert U_1\vert=\vert U_2\vert =t$, $t\ge 3$,  and $H$ is a pseudo sphere graph  $S_t$ (also called near pencil) defined as follows: we consider $t-1$ path graphs of order  $4$ and we identify one extreme of each one of the $t-1$ path graphs in one pole $a$ and all the other extreme vertices of the paths in a pole $b$. In particular, $S_3$ is the cycle graph $C_6$.

\item $H=(U_1\cup U_2,E)$ is the incidence graph of a non-degenerate projective plane of order $q$. In this case $H$ is a regular graph of degree $\delta_H=q+1$ and $\vert U_1 \vert=\vert U_2\vert =q^2+q+1$. Note that
$\vert U_1 \vert=\vert U_2\vert =\delta_H^2-\delta_H+1$.
\end{enumerate}


In  the case (P1)  the set $B=\{a,b\}$ composed by both poles of the pseudo sphere is a dominating set of $S_t$. Thus, $B$ is a local metric basis for $\langle v \rangle+S_r$ and $N_{S_t}(a)\cap N_{S_t}(b)=\emptyset$. Also, there are no local metric generators composed by two vertices at distance two, so the vertex $v$ does not belong to any local metric basis for $\langle v \rangle+S_t$ and, by Theorem \ref{mainTheorem} (i), we obtain that for any connected graph $G$ of order $n\ge 2$, $\dim_l(G\odot S_t)=2n$.

The rest of this section covers the study of case (P2), \textit{\textit{i.e.}}, the case where $H$ is the incidence graph  of a non-degenerate projective plane.

\begin{lemma}\label{gradosIguales}
For any    bipartite graph $H\not\cong S_t$ of diameter three and girth   six, $$\delta'(H)=\delta_H.$$
\end{lemma}

 \begin{proof}
 Let $x\in U_i$, $i\in \{1,2\}$. Since for any $y,z\in N_H(x)$  we have $N_H(y)\cap N_H(z)=\{x\}$, 
 we deduce that for any   $A\subseteq N_H(x)$,
 $$
 \left\vert  N_H^{(2)}(x) \right\vert=\vert U_i-\{x\}\vert\ge \displaystyle\left\vert\bigcup_{y\in A} (N_H(y)-\{x\})\right\vert=\sum_{y\in A}(\vert N_H(y)\vert-1)=\left(\delta_H-1 \right)\vert A\vert.
 $$
Therefore, since $\vert U_i \vert=\delta_H^2-\delta_H+1$, we have that $A\succ  N_H^{(2)}(x)$ if and only if $A=N_H(x)$.
 \end{proof}

\begin{lemma}\label{dominante}
Let $H=(U_1\cup U_2,E)\not\cong S_t$  be a  bipartite graph of diameter three and  girth   six. For any   local  metric basis $B$ of $K_1+H$,   either  $B\cap U_1=\emptyset$ or $B\cap U_2=\emptyset$.
\end{lemma}

\begin{proof}
  We proceed by contradiction. Suppose that $B_1=B\cap U_1\ne \emptyset$ and  $B_2=B\cap U_2\ne \emptyset$. We differentiate two cases.
\\
\\
\noindent{Case 1}: $ B_1\cup N_H(B_2)\ne U_1$ or $ B_2\cup N_H(B_1)\ne U_2$. We take, without loss of generality, $x\in U_1$ such that $x\not\in B_1\cup N_H(B_2)$. Since $B$ is a local metric basis for $K_1+H$ and $N_H(x)\cap B_2=\emptyset$, the set  $N_H(x)$ must be dominated by $B_1$. Moreover, since $H$ is a square free graph, for any $b\in B_1$  there exists only one vertex   $y_b\in  N_H(x)\cap N_H(b)$. Thus,
$\delta_H=\vert N_H(x)\vert \le \vert B_1\vert$. On the other hand, by Lemmas \ref{lemmaRadio3Bipartite} and \ref{gradosIguales} we have $\vert B\cap (U_1\cup U_2)\vert \le \delta_H$. Hence, the assumption $B_2=B\cap U_2\ne \emptyset$ leads to $\vert B_1\vert\le \delta_H-1$, which is a contradiction with  the fact that $\vert B_1\vert\ge \delta_H$.
\\
\\
\noindent{Case 2}: $ B_1\cup N_H(B_2)= U_1$ and $ B_2\cup N_H(B_1)= U_2$.
If $\vert B_1\vert = \vert B_2\vert=1$, then
$\delta_H^2-\delta_H+1=\vert U_1\vert = \vert B_1\cup N_H(B_2)\vert \le 1+\delta_H$, which is a contradiction for $\delta_H>2$. Thus, without loss of generality, we assume that $\vert B_2\vert \ge 2$. Let $a,b\in B_2$ and let $c\in U_1$ such that $\{c\}=N_H(a)\cap N_H(b).$ We define $B_1'=B_1\cup \{c\}$, $B_2'=B_2-\{a,b\}$  and $B'=B_1'\cup B_2'$. Note that $\vert B'\vert < \vert B\vert$.   We take two adjacent vertices $x,y$ such that $x\in U_1-B'_1$ and $y\in U_2- B_2'$.
Now, if $y\in \{a,b\} $, then $c\in B'$ distinguishes the pair $x,y$ and if $y\not\in \{a,b\} $, then there exists $y'\in B_1\subseteq B'$ such that $y'$ is adjacent to $y$. Thus, $B'$ is a local metric basis for $K_1+H$, which is a contradiction.

Since both cases lead to a contradiction, the proof is complete.
\end{proof}

\begin{lemma}\label{vpertenece}
Let $H\not\cong S_t$ be a  bipartite graph of diameter three and  girth  six. Then the vertex  of $K_1$ belongs to any local metric basis for   $K_1+H$.
\end{lemma}

\begin{proof} Let $B$ be a local metric basis for $\langle v \rangle+H$.  We proceed by contradiction. Suppose that $v\notin B$. By Lemmas \ref{lemmaRadio3Bipartite} and \ref{gradosIguales} we have $\vert B\vert \le \delta_H$.  By Lemma \ref{dominante} we can assume that $B\subset B_1$.  Now, if $\vert B\vert\le \delta_H-1$, then
$$\vert N_H(B)\vert=\left\vert \bigcup_{b\in B}N_H(b)\right\vert\le \sum_{b\in B}\vert N_H(b)\vert=(\delta_H-1)\delta_H<\vert U_2\vert,$$
which is a contradiction because if there exist two adjacent vertices $x,y$ such that  $x\in U_1-B$ and $y\in U_2-N_H(B)$, then the pair $x,y$ is not distinguished by the elements of $B$.   Hence, we conclude $\vert B\vert =\delta_H$.

Now, if there exists $a\in U_2$ such that $N_H(a)=B$, then the pair of adjacent vertices $a,v$ is not distinguished by the elements of $B$, which is a contradiction. Thus, let $b,b'\in B$, $a\in N_H(b)\cap N_H(b')$  and  $x_a\in N_H(a)-B$. Since $B$ is a local metric basis and $H$ is a square free graph, for every $y,z\in N_H(x_a)$, there exist two vertices  $b_y \in (B-\{b,b'\})\cap N_H(y)$ and $b_z \in (B-\{b,b'\})\cap N_H(z)$ such that $b_y\ne b_z$. Hence, $$\delta_H-1=\vert N_H(x_a)-a\vert \le \vert B-\{b,b'\}\vert =\delta_H-2,$$
which is a contradiction. Therefore, $v$ must belong to $B$.
\end{proof}

\begin{theorem}
Let $H \not\cong S_t$ be a  bipartite graph of diameter three and girth  six. Then for any connected graph $G$ of order $n\ge 2$,
$$\dim_l(G\odot H)=n\cdot \delta_H.$$
\end{theorem}
\begin{proof}
  By  Lemma \ref{vpertenece} we know that the vertex of $K_1$  belongs to every local metric basis for $K_1+H$, by Lemmas \ref{lemmaRadio3Bipartite} and \ref{gradosIguales} we have $\dim_l(K_1+H)=\delta_H+1$ and by Theorem \ref{mainTheorem} (ii) we conclude $\dim_l(G\odot H)=n\cdot \delta_H.$
\end{proof}


Let $\pi=(P,L)$ be a finite non-degenerate projective  plane of order $q$, where $P$ is the set of points and $L$ is the set of lines. Given two sets $P'\subset P$ and $L'\subset L$, we say that $P'\cup L'$ satisfies the property ${\cal G}$, if for any point $p_0$ and any line $l_0$ such that $p_0\in l_0$ we have
\begin{itemize}
\item there exists $p\in P'$ such that $p\in l_0$, or
\item there exists $l\in L'$ such that $p_o\in l$.
\end{itemize}

 We define $\Upsilon(\pi)=\min \{\vert P'\cup L'\vert$ such that  $P'\cup L'$ satisfies the property ${\cal G} \}$.

We have that if $H$ is the incidence graph of $\pi$, then a set $P'\cup L'$ satisfies the property ${\cal G}$ if and only if  $P'\cup L'\cup \{v\}$ is a local metric generator for $\langle v\rangle +H$. Therefore, according to Lemmas \ref{lemmaRadio3Bipartite}, \ref{gradosIguales} and  \ref{vpertenece}  we conclude $$\Upsilon(\pi)=\delta_H=q.$$
Note that if $P'\cup L'$ satisfies the property ${\cal G} $ and its cardinality is the  minimum among all the sets satisfying this property, then either $P'=\emptyset$ and $L'$ is the set of lines incident to one point or $L'=\emptyset$ and $P'$ is the set composed by all the points laying on one line.


\subsection{The minimum value}
As a direct consequence of Theorems \ref{mainTheorem} and \ref{Th-r=2Bipartite} we derive the following result.

\begin{remark}
For any connected graph $H$ of radius $r(H)\ge 3$ and any connected graph $G$ of order $n\ge 2$,
$$\dim_l(G\odot H)\ge 2n.$$
\end{remark}

In this section we study the limit case of the above bound for the case where $H$ is bipartite.

\begin{lemma}\label{dimK1+H=2vnotin}
If $H$ is a graph of radius three and $\dim_l(K_1+H)=2$, then the vertex of $K_1$ does not belong to any local metric basis for $K_1+H$.
\end{lemma}

\begin{proof}
  Let  $\{a,b\}$ be a local metric basis for $\langle v\rangle+H$. Since $r(H)=3$, no vertex of $H$ distinguishes every pair of adjacent vertices of $H$. Thus, $a\ne v$ and $b\ne v$.
\end{proof}


\begin{theorem}\label{ThdimG0H=2n-R=3}
Let  $H=(U_1,U_2,E)$ be  a bipartite  graph of radius three and let $G$ be a connected graph of order $n$.
Then $\dim_l(G\odot H)=2n$ if and only if
$\dim_l(K_1+H)=2$ or for some $i\in \{1,2\}$, there exist  $a,b\in U_i$ such that $N_H(a)\cup N_H(b)=U_j$, where $j\in \{1,2\}-\{i\}$.
\end{theorem}

\begin{proof}
 By Theorem \ref{mainTheorem} we know that $\dim_l(G\odot H)=2n$  if and only if either $\dim_l(\langle v\rangle+H)=2$ and $v$ does not belong to any local metric basis for $\langle v\rangle+H$ or $\dim_l(\langle v\rangle+H)=3$ and there exists a local metric basis $B$  of $\langle v\rangle+H$ such that $v\in B$.

If $\dim_l(\langle v\rangle+H)=2$, then we are done (note that by Lemma \ref{dimK1+H=2vnotin} we have that $v$ does not belong to any local metric basis for $\langle v\rangle+H$).


Let $B=\{a,b,v\}$ be a local metric basis  of $\langle v\rangle+H$. Since $v\in B$, we have $N_H(a)\cap N_H(b)\ne \emptyset$. So, $a$ and $b$ must belong to the same color class, set $a,b\in U_1$. Hence, if there exists $y\in U_2-(N_H(a)\cup N_H(b))$, then for every $x\in N_H(y)$, the pair $x,y $ is not distinguished in $\langle v\rangle+H$ by the elements of $B$, which is a contradiction and, as a consequence, $N_H(a)\cup N_H(b)=U_2$.

Conversely, if there exists  $a,b\in U_i$ such that $N_H(a)\cup N_H(b)=U_j$, where $j\in \{1,2\}-\{i\}$, then for every $y\in U_j$ and $x\in N_H(y)$, the pair $x,y$ is distinguished by $a$ or by $b$. So,  $\{a,b,v\}$ is a local metric generator for  $\langle v\rangle+H$ and, as a consequence, $ \dim_l(\langle v\rangle+H)\le 3$. Therefore, either $\dim_l(\langle v\rangle+H)= 2$ or $\{a,b,v\}$ is a local metric basis   of $\langle v\rangle+H$.
\end{proof}

Consider the following  decision problem. The input is an arbitrary bipartite graph $H=(U_1\cup U_2,E)$  of radius three. The problem consists in deciding whether $H$ satisfies $\dim_l(K_1+H)=2$, or not. According to the next remark we deduce that the time complexity of this decision problem is at most $O(  \vert U_1\vert^2\vert U_2\vert^2)$. Although this remark is straightforward, we include the proof for completeness.

\begin{remark}\label{lemadimK1+H=2vnotinR=3}
Let  $H=(U_1,U_2,E)$ be  a bipartite  graph of radius three. Consider the following statements:
\begin{enumerate}[{\rm (i)}]
\item For some $i\in \{1,2\}$, there exist  $a,b\in U_i$ such that $\{N_H(a), N_H(b)\}$ is a partition of $U_j$, where $j\in \{1,2\}-\{i\}$.
\item There exist two vertices $a\in U_1$ and $b\in U_2$ such that for every edge $xy\in E$, where $x\in U_1$ and $y\in U_2$, it follows $y\in N_H(a)$ or   $x\in N_H(b)$.
\end{enumerate}
 Then  $\dim_l(K_1+H)=2$ if and only if {\rm (i)} or {\rm (ii)} holds.
\end{remark}

\begin{proof}
 We first note that since $K_1+H$ is not bipartite, Theorem \ref{The1Zhang} leads to $\dim_l(K_1+H)\ge 2$.

(Sufficiency)  If (i) holds, then $\{a,b\}\succ U_j$ and $N_H(a)\cap N_H(b)=\emptyset$. Hence, $\{a,b\}$ is a local metric basis   of $K_1+H$ and, as a consequence, $\dim_l(K_1+H)=2$.

Now, if (ii)  holds, it is is straightforward that  $\{a,b\}$ is a local metric basis  of $K_1+H$ and, as a consequence, $\dim_l(K_1+H)=2$.

(Necessity) Let $\{a,b\}$ be a local metric basis for of $\langle v\rangle+H$. By Lemma  \ref{dimK1+H=2vnotin} we know that $v\not \in \{a,b\}$. Then we have two possibilities.
\\
\\
\noindent Case 1. $a$ and $b$ belong to the same color class of $H$, say $a,b\in U_1$.  Since for every $x\in V(H)$ the pair $x,v$ must be distinguished by $a$ or by $b$, we conclude that $N_H(a)\cap N_H(b)=\emptyset$. Also, since every pair of adjacent vertices $x\in U_1$ and $y\in U_2$ must be distinguished by $a$ or by $b$, we conclude that $y\sim a$ or $y\sim b$ and, as a result,   $\{a,b\}\succ U_2$. Hence, we conclude that $\{N_H(a), N_H(b)\}$ is a partition of $U_2$.
\\
\\
\noindent Case 2: $a$ and $b$ belong to different color classes of $H$, say $a\in U_1$ and $b\in U_2.$ Since $\{a,b\}$ is a local metric basis for $\langle v\rangle+H$, for every edge $xy\in E$, where $x\in U_1$ and $y\in U_2$, it follows $y\in N_H(a)$ or   $x\in N_H(b)$.
\end{proof}

Note that if  $H=(U_1\cup U_2,E)$ is a bipartite graph of diameter $D(H)=3$, then for any $i\in \{1,2\}$ and $x,y\in U_i$ we have $N_H(x)\cap N_H(y)\ne \emptyset$. Hence, we deduce the following consequence of Remark \ref{lemadimK1+H=2vnotinR=3}.

\begin{corollary}
Let $H$ be a bipartite graph where $D(H)=r(H)=3$.
If  $B=\{a,b\}$ is a local metric basis for $K_1+H$, the $a$ and $b$ belong to different color classes.
\end{corollary}

Other direct consequence of Remark \ref{lemadimK1+H=2vnotinR=3} is the following.

\begin{corollary}
Let  $H=(U_1,U_2,E)$ be  a bipartite  graph of radius three.
If for some $i\in \{1,2\}$, there exist  $a\in U_i$ such that $\delta_H(a)=\vert U_j\vert-1$, where $j\in \{1,2\}-\{i\}$, then $\dim_l(K_1+H)=2$.
\end{corollary}

\subsection{Closed formulae for $\dim_l(G\odot H)$ when $H$ is a tree of radius three}
In order to study the particular case when $H$ is a tree of radius three, we introduce   the following additional notation.
Let $T$ be a tree of radius three. For the particular case when $C(T)=\{u\}$  we consider  the forest $F_u=\cup_{w\in N_T(u)} T_{w}$ composed of all the rooted trees $T_w=(V_w,E_w),$ of root $w\in N_T(u)$,  obtained by removing the central vertex $u$ from $T$.
The height of $T_w$ is $h_w=\max_{x\in V(T_w)}\{d(w,x)\}$.
We denote by $\varsigma(T)$ the number of  trees in $F_u$ with $h_w$ equal to two, \textit{i.e.},  $\varsigma(T)=|S(T)|,$ where
$$S(T)=\{w\in N_T(u):\; h_w=2\}.$$
Note that if $h_w\ne 1$, for every $w\in N_T(u)$, then $\varsigma(T)=\delta'(T)$. So, as the following result shows,   the bound $\dim_l(G\odot T)\le n\cdot \delta'(T)$ is tight.

\begin{theorem}
Let $T$ be a tree of radius three and center $C(T)$.
The following {assertions} hold for any connected graph $G$ of order $n\ge 2$.
\begin{enumerate}[{\rm (i)}]
\item If $|C(T)|=2,$ then $\dim_l(G\odot T)=2n$
\item If $C(T)=\{u\},$ then
    $$\dim_l(G\odot T)=\left\{
                               \begin{array}{ll}
                                   n\cdot (\varsigma(T)+1),& \mbox{if there exists}\; w\in N_T(u) \;\mbox{\rm such that}\;  h_w=1,\\
                                  n\cdot \varsigma(T),& \mbox{otherwhise}.\\
                               \end{array}
                             \right.
    $$
\end{enumerate}
\end{theorem}

\begin{proof}
It is well-known that the center of a tree consists of either a single vertex or two adjacent vertices.

We first consider the case where $C(T)$ consists of  two adjacent vertices, say $C(T)=\{u',u''\}$. Note that in this case, if we remove the edge $\{u',u''\}$ from $T$, we obtain two rooted trees $T'=(V',E')$ and $T''=(V'',E'')$, with roots  $u'$ and $u''$, respectively, where the distance from the root to the leaves is at most two. Hence, in $K_1+T$ every pair of adjacent vertices $x,y\in V'$ is distinguished by $u'$ and every pair of adjacent vertices $x,y\in V''$ is  distinguished by $u''$. Also, for every $x\in V'-\{u'\}$ the pair $v,x$ is distinguished by $u''$ and for every $x\in V''-\{u''\}$, the pair $v,x$ is distinguished by $u'$, where $v$ is the vertex of $K_1$. So, $C(T)$ is a local metric generator for $K_1+T$. Hence, $\dim_l(K_1+T)\le 2$ and, since $K_1+T$ is not bipartite, by Theorem  \ref{The1Zhang} we conclude that $\dim_l(K_1+T)=2$. Now, in this case, if  the vertex of $K_1$ belongs to a  local metric basis for $K_1+T$, then there exists $z\in V(T)$ such that  $z$ distinguishes any pair of adjacent vertices $x,y\in V(T)$, and as a consequence $r(T)\le 2$, which is a contradiction. Thus,  we conclude that the vertex of $K_1$ does not belong  to any local metric basis for $K_1+T$. Therefore, as a consequence of Theorem \ref{mainTheorem} (i)  we obtain $\dim_l(G\odot T)=2n$.

Now let us consider the case where the center of $T$ consists of a single vertex, say $C(T)=\{u\}$. Let $B$ be a local metric basis for $K_1+T$. We first note that for every rooted tree $T_w=(V_w,E_w)$ of height two we have $|B\cap V_w|=1$, due to the fact that in $K_1+T$ the vertex  $w\in N_T(u)$  distinguishes  every pair of adjacent vertices $x,y\in V_w$ and no vertex of $V(K_1+T)-V_w$ distinguishes  a pair of adjacent vertices where one vertex is a leaf. Hence, $\dim_l(K_1+T)\ge \varsigma(T).$
Now we differentiate the following cases.
\\
\\
\noindent
Case 1.  There exists $w\in N_T(u)$   such that $h_w=1$.  In this case, the subgraph of $T$ induced by the set  $X=\displaystyle\cup_{ h_w\le 1 }V_w\cup \{u\}$ is a tree of root $u$ and height two. Hence, as above we conclude that $|B\cap X|= 1$. So, $\dim_l(K_1+T)\ge \varsigma(T)+1.$ In order to show that the set  $A=\{u\}\cup S(T)$ is a local metric basis for $K_1+T$ we only need to observe that $N_{T}(w)\cap N_T(u)=\emptyset$ and, as a consequence, for every $x\in V(T)$ the pair $x,v$ is distinguished by  some $z\in A$.
Thus, $\dim_l(K_1\odot T)=\varsigma(T)+1.$

Moreover, since for every metric basis $A$ of $K_1+T$ we have  $|A\cap X|= 1$ and for every rooted tree $T_w=(V_w,E_w)$ of height two,  $|A\cap V_w|=1$, we conclude that the vertex of $K_1$ does not belong to any local metric basis for $K_1+T$.  Therefore, as a consequence of Theorem \ref{mainTheorem} (i)   we obtain $\dim_l(G\odot T)=n(\varsigma(T)+1)$.
\\
\\
\noindent Case 2. For every $w\in N_T(u)$,  $h_w\ne 1$. In this case we define $$\varphi(T_w)=|\{z\in N_{T_w}(w):\; \delta_T(z)\ge 2\}|.$$
Suppose there exists $w_i\in N_T(u)$ such that $\varphi(T_{w_i})=1$. With this assumption we define $$A'=\{z\}\cup S(T)-\{w_i\},$$ where $z\in V_{w_i}$ and  $\delta_T(z)\ge 2$. Note that every pair of adjacent vertices $x,y\in \{u\}\cup  V_{w_i}$ is distinguished by $z$. So, by analogy to Case 1 we show that $A'$ is a local metric basis for $K_1+T$ and the vertex of $K_1$ does not belong to any local metric basis for $K_1+T$.  Therefore, as a consequence of Theorem \ref{mainTheorem} (i)   we obtain $\dim_l(G\odot T)=n\cdot \varsigma(T)$.

On the other hand, if for every  $w\in S(T)$ it follows  $\varphi(T_w)\ge 2$, then $w$ is the only vertex of $V_w$ which distinguishes every pair of adjacent vertices $x,y\in V_w$. Thus, in such a case $S(T)$ is a subset of any local metric basis for $K_1+T$ and, as a consequence, the only two local metric basis for $K_1+T$ are $\{u\}\cup S(T)$ and $\{v\}\cup S(T)$. Therefore, as a consequence of Theorem \ref{mainTheorem} (ii)   we obtain $\dim_l(G\odot T)=n\cdot \varsigma(T)$.
\end{proof}

\end{document}